\documentclass[12pt,a4paper,reqno]{amsart}  
\usepackage{amsmath}
\usepackage{amsfonts}
\usepackage{amssymb}
\numberwithin{equation}{section}

\newcommand\R{{{\mathbf R}}}
\newcommand\C{{{\mathbf C}}}

\newcommand\Z{{{\mathbf Z}}}
\newcommand\F{{{\mathbf F}}}

\newcommand\bigO{{\mathbf{O}}}

\newcommand\eps{\varepsilon}

\theoremstyle{plain}
  \newtheorem{theorem}[subsection]{Theorem}

  \newtheorem{proposition}[subsection]{Proposition}
  
  \newtheorem{lemma}[subsection]{Lemma}

\theoremstyle{remark}
  \newtheorem{remark}[subsection]{Remark}

\theoremstyle{definition}

\include{psfig}

\begin{document}

\title{The sum-product phenomenon in arbitrary rings}

\author{Terence Tao}
\address{Department of Mathematics, UCLA, Los Angeles CA 90095-1555}
\email{tao@@math.ucla.edu}

\begin{abstract}
The \emph{sum-product phenomenon} predicts that a finite set $A$ in a ring $R$ should have either a large sumset $A+A$ or large product set $A \cdot A$ unless it is in some sense ``close'' to a finite subring of $R$.  This phenomenon has been analysed intensively for various specific rings, notably the reals $\R$ and cyclic groups $\Z/q\Z$.  In this paper we consider the problem in arbitrary rings $R$, which need not be commutative or contain a multiplicative identity.  We obtain rigorous formulations of the sum-product phenomenon in such rings in the case when $A$ encounters few zero-divisors of $R$.  As applications we recover (and generalise) several sum-product theorems already in the literature.
\end{abstract}

\maketitle

\section{Introduction} 

\subsection{The sum-product phenomenon}

Let $R = (R,0,+,-,\cdot)$ be a ring (which need not be commutative, and need not contain a multiplicative identity $1$).  Given any sets $A, B \subset R$, we define the \emph{sum set}
$$ A + B := \{ a+b: a \in A, b \in B \}$$
the \emph{difference set}
$$ A - B := \{ a-b: a \in A, b \in B \}$$
and the \emph{product set}
$$ A \cdot B := \{ ab: a \in A, b \in B \}.$$
We also define the iterated sum sets and iterated product sets
$$ nA := A + \ldots + A; \quad A^n := A \cdot \ldots \cdot A$$
for $n \geq 1$, with the convention $0A:=\{0\}$.  We also write $r \cdot A := \{r\} \cdot A$, $A \cdot r := A \cdot \{r\}$, and $r+A=A+r := A + \{r\}$ for any ring element $r \in R$.  If $A$ is a finite set, we use $|A|$ to denote the cardinality of $A$.

We let $R^*$ denote the collection of \emph{non-zero-divisors} of $R$, i.e. the elements $r \in R$ such that $ra, ar \neq 0$ whenever $a \neq 0$.  Observe that $R^*$ is closed under multiplication, and that $|r \cdot A| = |A \cdot r| = |A|$ for any $r \in R^*$ and finite $A$.

If $A$ is a finite subring of $R$, then clearly $|A+A| = |A|$ and $|A \cdot A| \leq |A|$, with equality holding in the latter case if $A$ contains at least one non-zero-divisor.  The same holds if $A$ is a dilate $r \cdot R$ or $R \cdot r$ of a commutative ring for some non-zero-divisor $r$.  The remarkable \emph{sum-product phenomenon} asserts a robust converse to this simple observation in many cases: very roughly speaking, it asserts that if $A$ is a finite non-empty subset of a suitable ring $R$ with $A+A$ and $A \cdot A$ both having size comparable to $A$, then $A$ should be ``very close'' to a ring (or a dilate of a ring), for instance $A$ might be contained in a finite ring (or dilate of a finite ring) of size comparable to $A$.

\subsection{Prior results}

The first rigorous demonstration of the sum-product phenomenon was by Erd\H{o}s and Szemer\'edi \cite{erdos-sumproduct} in the ring of integers $\Z$, in which of course there are no non-trivial finite subrings.  In this context they showed that for any finite non-empty set $A \subset \Z$, one had $|A+A| + |A \cdot A| \geq c |A|^{1+\eps} |A|$ for some $c, \eps > 0$, and conjectured that one can in fact take $\eps$ arbitrarily close to $1$.  There is a substantial further literature on this problem \cite{ford}, \cite{nathanson-product}, \cite{elekes}, \cite{chen}, \cite{nathanson}, \cite{elru}, \cite{enr}, \cite{chang-comp}, \cite{chang-sumproduct}, \cite{chang-gafa}, \cite{bourgain-chang-comp}, \cite{bourgain-chang}, \cite{chang-distinct}, \cite{chang-distinct2}, \cite{tao-vu}, \cite{chang-survey}, \cite{konlaba}, \cite{chang-soly}, \cite{soly}; for instance, Solymosi \cite{soly} recently showed that $\eps$ can be taken arbitrarily close to $1/3$ (and that the integers can be replaced with the reals).  See \cite{chang-survey} for a brief survey of some other recent results in this direction.  For a continuous version of the sum-product phenomenon in $\R$ (related to the Erd\H{o}s-Volkmann ring conjecture \cite{erdos-ring}  first solved in \cite{edgar}), see \cite{borg-ring}; this result has applications to geometric measure theory \cite{katz-falconer} and the theory of invariant measures \cite{bfl} and spectral gaps \cite{su2}.

For the complex numbers $\C$, Solymosi also showed in \cite{soly-complex} that one can take $\eps$ arbitrary close to $1/4$.  An earlier result of Chang \cite{chang} takes $\eps$ arbitrarily close to $1/54$, but allows $R$ to be either $\C$ or the quarternion algebra, and also gives similar results (with non-explicit values of $\eps$) for finite-dimensional division algebras over $\R$; see also another paper of Chang \cite{chang-semisimple} which achieves a similar result for semi-simple commutative Banach algebras over $\R$ or $\C$, and in particular in the infinite product spaces $\R^\Z$ and $\C^\Z$.  On the other hand, it was observed in \cite{chang-semisimple} that $\eps$ cannot exceed $1 - \frac{\log 2}{\log 3}$ in those spaces.

When $R$ is a matrix ring over $\R$ or $\C$, it was shown in \cite{chang-matrix} that $|A+A| + |A \cdot A| \geq f(|A|) |A|$ whenever $(A-A) \backslash R^* = \{0\}$, for some function $f(x)$ depending on $R$ which goes to infinity as $x \to \infty$; in the case when $A$ consisted entirely of real symmetric matrices, one could take $f(x) = c x^\eps$ for some $c,\eps > 0$ depending on $R$ (as in previous sum-product results).  In the slightly different context of multiplicative matrix groups rather than rings, it was shown in \cite{chang-sl2} that when $A \subset \mathbb{SL}_3(\Z)$ was a finite non-empty set, then $|A \cdot A \cdot A| \geq c |A|^{1+\eps}$ unless a large subset of $A$ was contained in a coset of a nilpotent subgroup; a similar result was also established for $R = \mathbb{SL}_2({\C})$ in the same paper, in which the nilpotent subgroup was now abelian.

Now we turn to the case of finite characteristic.
In the case when $R$ is a finite field of prime order, which has no non-trivial subrings other than the full ring $R$, it was shown in \cite{bkt}, \cite{konyagin}, \cite{konyagin2} that for every $\delta > 0$ there exists $\eps > 0$ and $c > 0$ such that $|A+A| + |A \cdot A| \geq c |A|^{1+\eps}$ for all finite non-empty $A \subset R$ with $|A| \leq |R|^{1-\delta}$.  (Some upper bound on $|A|$ is of course necessary since $|A+A|$ and $|A \cdot A|$ clearly cannot exceed $|R|$.)  This estimate has numerous applications to exponential sums, number theory, combinatorics, and computer science; see \cite{borg-diffie-comp}, \cite{borg-diffie}, \cite{borg-exp}, \cite{borg-mordell}, \cite{bourgain-survey}, \cite{borg-more}, \cite{borg-arith}, \cite{bg}, \cite{bgs}, \cite{borg-extract}.  In \cite{garaev}, it was shown that one could take $\eps$ arbitrarily close to $1/14$ if $\delta > 6/13$; a slight variant of this argument in \cite{katz} showed that one can take $\eps$ arbitrarily close to $1/13$ if $\delta > 1/2$, and in \cite{boga} it was shown that one can take $\eps$ arbitrarily close to $1/12$ if $\delta > 11/23$ and $A+A$ is replaced by $A-A$.  In \cite{garaev2} it was shown that one could take $\eps=\delta/2$ if $\delta < 1/3$.  Variants of these results for elliptic curves, or for exponentiated versions of the sum-product problem in finite fields, were obtained in \cite{shpar1}, \cite{shpar2}, \cite{shpar3}; generalisations to other polynomials than the product operation were considered in \cite{vu}.  Also, in \cite{wood} a general embedding theorem was established which allowed one to transfer sum-product type theorems in finite fields to commutative integral domains of characteristic zero.

When $R$ is a more general finite field, the situation is more complicated due to the presence of non-trivial rings of intermediate size, namely the subfields of $R$ and their dilates.  In \cite[Theorem 4.3]{bkt} it was shown that if $A \subset R$ is non-empty with $|A+A| + |A \cdot A| \leq K |A|$ and $|A| \geq |R|^\delta$ then there must be a subfield $F$ of $R$ of size $|F| \leq K^{O_\delta(1)} |A|$ and an invertible element $a$ of $F$ such that $A \subset a \cdot F + \bigO( K^{O_\delta(1)} )$, where $\bigO( N )$ denotes a non-empty set in $R$ of cardinality $O(N)$, and the subscripts in the $O()$ notation indicate that the implied constants can depend on $\delta$.  In the converse direction, observe that if $A \subset a \cdot F + \bigO( K^{O(1)} )$ then $|A+A| + |A \cdot A| = O( K^{O(1)} |F| )$, so this result is sharp up to polynomial factors in $K$ and dependence on $\delta$.  In \cite[Theorem 2.55]{tao-vu} the hypothesis $|A| \geq |R|^\delta$ was removed, with the result also extending to infinite fields $R$ (thus recovering in particular some of the sum-product theory in $\Z$, $\R$, and $\C$).  In \cite{katz2} the following explicit variant was established: if $|A+A| + |A \cdot A| \leq c |A|^{1+\frac{1}{48}}$ for some sufficiently small absolute constant $c > 0$, then there exists a set $A' \subset A$ with $|A'| \geq |A|^{1-\frac{1}{48}}$ and a subfield $F$ of $R$ of size $|F| \leq |A'|^2$ such that $A' \subset a \cdot F + b$ for some $a,b \in R$.  Other results of sum-product type in finite fields (or bounded dimensional vector spaces over such fields) in the case when $A$ is large (e.g. $|A| > |F|^{1/2}$) were obtained in \cite{hart}, \cite{hart2}, \cite{covert}.

The case of more general cyclic rings $R = \Z/q\Z$ than the fields of prime order is considered in \cite{bourgain-chang-zq}, \cite{bourgain-survey}, \cite{bgs}, \cite{borg-sump-exp}; in particular, it was shown in \cite{borg-sump-exp} that if $A \subset R$ and $1 \leq |A| \leq |R|^{1-\delta_1}$, and $|\pi_{q_1}(A)| \geq q_1^{\delta_2}$ for all $q_1 | q$ with $q_1 \geq q^{\eps}$ for some sufficiently small $\eps = \eps(\delta_1,\delta_2) > 0$, where $\pi_{q_1}: R \to \Z/q_1 \Z$ is the projection homomorphism, then $|A+A| + |A \cdot A| \geq q^{\delta_3} |A|$ for some $\delta_3 = \delta_3(\delta_1,\delta_2) > 0$. Similar results for rings such as $R = \Z/p\Z \times \Z/p\Z$ also appear in \cite{borg-diffie-comp}, \cite{borg-diffie}, \cite{borg-mordell}.  Applications of these estimates to exponential sums appear in \cite{borg-chang-exp}, \cite{borg-finite-exp}, \cite{borg-sump-exp}, \cite{borg-exponential}.

Finally, we mention the result of Helfgott \cite{helf} that shows that if $A \subset \mathbb{SL}_2({\mathbf F}_p)$ with $|A| \leq p^{3-\delta}$, and $A$ is not contained in any proper subgroup, then $|A \cdot A \cdot A| \geq c |A|^{1+\eps}$ for some $c,\eps > 0$ depending only on $\delta > 0$.  This result (and variants for other groups, including continuous groups such as $SU(2)$) has applications to expander graphs, sieving, and diophantine approximation: see \cite{bg}, \cite{bgs}.

\subsection{New results}

In this paper we study the sum-product phenomenon in arbitrary rings $R$, which need not be commutative or to contain a multiplicative identity $1$.  In doing so one must make some sort of assumption to avoid too many zero divisors; in the most extreme case, when the product operation is identically zero, then $A \cdot A$ is always just $\{0\}$.  There does not appear to be any canonical way to get around this issue; for us, it will be convenient to make two (related) non-degeneracy assumptions.  The first is that $A \cdot A$ is not much smaller than $A$; the other is that $A-A$ does not have an extremely large number of zero-divisors.  These assumptions seem to be reasonable in situations in which the set of zero-divisors in $R$ is very sparse; it would be of interest to weaken our hypotheses to handle rings with many zero-divisors\footnote{In particular, our results here do not fully recover the results in $\Z/q\Z$ mentioned above in the case when $q$ has many small divisors.  Indeed, it seems to the author that this case requires a genuinely multiscale analysis and so may in fact be beyond the purely elementary approach used here.}, but we will not do so here.

Our results are of the following general form: if $A$ is a finite non-empty subset of a ring $R$ for which certain additive and multiplicative combinations of $A$ are small, and $A$ is non-degenerate in the sense described above, then $A$ can be efficiently contained in a ring, or a slight modification of a ring.  

The simplest case is if we assume that $A + A \cdot A$ is small (comparable to $A$ in size).  Examples of such sets include finite subrings $S \subset R$ of $R$, as well as dense subsets of such rings.  Our first result, roughly speaking, asserts (in the non-degenerate case) that these are in fact the only such sets with this property.

\begin{theorem}[Inhomogeneous sum-product theorem]\label{sum1-thm} Let $R$ be a ring, and let $A \subset R$ be finite and non-empty.  Suppose that $|A + A \cdot A| \leq K |A|$ and $|A \cdot A| \geq |A|/K$ for some $K \geq 1$.  Then at least one of the following holds for some absolute constant $C > 0$:
\begin{itemize}
\item[(i)] ($A-A$ has many zero divisors) We have $|(A-A) \backslash R^*| \geq C^{-1} K^{-C} |A|$; or
\item[(ii)] (Subring structure) There exists a finite subring $S$ of $R$ such that $A \subset S$ and $|S| \leq C K^C |A|$. 
\end{itemize}
\end{theorem}

Note in the converse direction that if $A \subset S$ for some finite subring $S$ and $|S| \leq K |A|$, then $|A + A \cdot A| \leq |S| \leq K |A|$, and so the conclusion here is reasonably sharp up to polynomial losses.  Also, the hypothesis $|A \cdot A| \geq |A|/K$ is automatic if $A$ contains at least one non-zero-divisor.

The hypothesis that $A + A \cdot A$ is small is \emph{inhomogeneous} in the sense that it is not preserved by dilations $A \mapsto r \cdot A$ (of course, the property of being a subring is also not homogeneous).  Let us now consider the homogeneous case when $A \cdot A - A \cdot A$ is small; examples of such sets $A$ include dilates $a \cdot S$ of finite rings $S$ for an invertible element $a$ (assuming $R$ has an identity), as long as $a$ normalises $S$ in the sense that $a \cdot S = S \cdot a$.  We can obtain a converse to this claim, similarly to Theorem \ref{sum2-thm}, under the additional assumption that $A$ contains an invertible element:

\begin{theorem}[Homogeneous sum-product theorem with invertible element]\label{sum2-thm} Let $R$ be a ring with identity, and let $A \subset R$ be finite and non-empty.  Suppose that $|A \cdot A - A \cdot A| \leq K |A|$ and $|A \cdot A| \geq |A|/K$ for some $K \geq 1$.  Suppose also that $A$ contains an invertible element $a$.  Then at least one of the following holds for some absolute constant $C$:
\begin{itemize}
\item[(i)] ($A-A$ has many zero divisors) We have $|(A-A) \backslash R^*| \geq C^{-1} K^{-C} |A|$; or
\item[(ii)] (Subring structure) There exists a finite subring $S$ of $R$ such that $A \subset a\cdot S=S\cdot a$ and $|S| \leq C K^C |A|$.
\end{itemize}
\end{theorem}

Now we consider the homogeneous case in more generality, when $R$ need not contain an identity and $A$ need not contain an invertible element.  In this case, there are more examples of sets that have good additive and multiplicative properties.  For instance, if $S$ is a finite ring and $S[t] := \{ \sum_{n=0}^d a_n t^n: a_0,\ldots,a_d \in S; d \geq 0\}$ is the polynomial ring generated by $S$ and a formal variable $t$ that commutes with $S$, then the set $A := S \cdot t$ in the ring $S[t] \cdot t$ of polynomials with no constant term is such that $|A+A|, |A \cdot A|, |A \cdot A - A \cdot A| \leq |A|$, but $S[t] \cdot t$ contains no non-trivial finite subrings.  Of course, this obstruction is artificial in nature because the ambient ring $S[t] \cdot t$ can be embedded in a larger ring, such as the polynomial ring $S[t]$ or the Laurent polynomial ring $S(t)$ generated by $S$, $t$, and a formal inverse $t^{-1}$ to $t$, which does contain finite subrings, in particular the ring $S$ of constant polynomials, and once we embed into this larger ring, then $A$ does become efficiently captured by a dilate of a subring.

A generalisation of the above example occurs when one has a finite ring $S$ with an (outer) ring automorphism $\phi: S \to S$.  Then one can form the twisted polynomial ring $S[t]_\phi$ generated by $S$ and a formal variable $t$ with the relations $ta = \phi(a) t$ for all $a \in S$, or the larger twisted Laurent polynomial ring $S(t)_\phi$ generated by $S$, a formal variable $t$, and its formal inverse $t^{-1}$ with the relations $ta = \phi(a) t$ (or equivalently $a t^{-1} = t^{-1} \phi(a)$) for all $a \in S$.  Then, as before, the set $A := S \cdot t = t \cdot S$ in the ring $S[t]_\phi \cdot t$ is such that $|A+A|, |A \cdot A|, |A \cdot A - A \cdot A| \leq |A|$, but $A$ is not contained efficiently in a dilate of a subring until one embeds that ring into $S[t]_\phi$ or $S(t)_\phi$.

We can now present a converse statement.

\begin{theorem}[Homogeneous sum-product theorem in general]\label{sum3-thm} Let $R$ be a ring, and let $A \subset R$ be finite and non-empty. Suppose that $|A \cdot A - A \cdot A| \leq K |A|$ and $|A \cdot A| \geq |A|/K$ for some $K \geq 1$.  Then at least one of the following holds for some absolute constant $C > 0$:
\begin{itemize}
\item[(i)]  ($A-A$ has many zero divisors) We have $|(A-A) \backslash R^*| \geq C^{-1} K^{-C} |A|$; or
\item[(ii)] (Ring structure in a Freiman model) There exists a finite ring $R_0$ of cardinality between $|A|$ and $CK^C |A|$, an outer automorphism $\phi: R_0 \to R_0$ of that ring, and embeddings $\iota_n: \langle A^n \rangle \to R_0 \cdot t^n$ from the additive group $\langle A^n \rangle$ generated by the $n$-fold product set $A^n := A \cdot \ldots \cdot A$ to the degree $n$ component $R_0 \cdot t^n$ of the twisted polynomial ring $R_0[t]_\phi$ such that the $\iota_n$ are additive homomorphisms, and that $\iota_n(g_n) \iota_m(g_m) = \iota_{n+m}(g_n g_m)$ for all $n,m \geq 1$ and all $g_n \in \langle A^n \rangle$ and $g_m \in \langle A^m \rangle$.
\end{itemize}
\end{theorem}

Observe that if $A, R_0, \phi, \iota_n$ are as in (ii), then $A \cdot A - A \cdot A \subset \langle A^2 \rangle$ has cardinality at most $|R_0| \leq CK^C |A|$, so the conclusion (ii) is efficient up to polynomial losses.  It may be possible to combine all the separate embeddings $\iota_n: \langle A^n \rangle \to R_0 \cdot t^n$ into a single embedding of the ring generated by $A$ into some suitable ring extension of $R_0$, but we were not able to achieve this, and in any event the ``Freiman-type'' or ``graded ring homomorphism'' collection of embeddings $\iota_n$ (somewhat analogous to the embeddings $\iota_n: n A \to n B$ associated to a Freiman isomorphism $\iota_1: A \to B$ of order at least $2n$, see \cite[Section 5.3]{tao-vu}) suffice for the purposes of studying (homogeneous) iterated sum and product sets of $A$.

Finally, we return to the traditional hypotheses for the sum-product phenomenon, in which we wish to classify the cases in which $A+A$ and $A \cdot A$ are small.  For this we record a non-commutative version of the ``Katz-Tao lemma'' originating in \cite{katz-falconer} (see also \cite{bkt}) and then simplified in \cite{borg-mordell} (see also \cite[Lemma 2.53]{tao-vu}), which lets us pass from a set $A$ with $A+A$ and $A \cdot A$ both small, to a slightly smaller set $A'$ with $A' \cdot A' - A' \cdot A'$ both small, as long as we first throw away all zero divisors:

\begin{lemma}[Katz-Tao lemma]\label{ktl}  Let $R$ be a ring, and let $A \subset R^*$ be a finite non-empty set of non-zero-divisors such that $|A+A|, |A \cdot A| \leq K|A|$ for some $K \geq 1$.  Then one of the following holds:
\begin{itemize}
\item[(i)] ($A-A$ has many zero divisors) $|(A-A) \backslash R^*| \geq C^{-1} K^{-C} |A|$ for some absolute constant $C > 0$;  
\item[(ii)] (Existence of good subset) There exists a subset $A'$ of $A$ such that $|A'| \geq |A|/2K$ and $|A' \cdot A' - A' \cdot A'| = O(K^{O(1)} |A'|)$.
\end{itemize}
\end{lemma}

The commutative version of this claim (without the need for the option (i)) was established in \cite[Lemma 2.53]{tao-vu}.  We do not know if the option (i) can similarly be removed in the noncommutative setting; one may need to first strengthen the bound $|A \cdot A| \leq K|A|$ to $|A \cdot A \cdot A| \leq K|A|$ for this (cf. \cite{tao-noncomm}).

Our methods here are elementary, relying entirely on Pl\"unnecke-Ruzsa sum set theory (see e.g. \cite[Chapter 2]{tao-vu} for a detailed treatment of this topic), and an analysis of certain key sets $S_r$ defined in \eqref{sr}, which roughly speaking contain the elements $x \in R$ for which $x \cdot A$ and $r \cdot A$ are ``parallel''.  

These elementary methods are able to treat sum-product estimates in any ring that does not have too many ``scales''; the sum-product phenomenon in multi-scale situations such as continuous subsets of $\R$ \cite{borg-ring}, or subsets of $\Z_p^m$, $\Z_p[t]/\langle t^m\rangle$ or $\Z_{p^m}$ \cite{borg-finite-exp} for large $m$ seem to require a more sophisticated analysis which we do not study further here, due to the presence of many zero-divisors.  By the same token, however, due to the soft and elementary nature of our methods, our results do not distinguish between finite or infinite rings, or between zero characteristic and positive characteristic.

\subsection{Organisation of the paper}

In Section \ref{sumsetsec} we recall the (standard) sum set estimates from Pl\"unnecke-Ruzsa theory that we shall need, together with a proof of Lemma \ref{ktl} in Section \ref{ktl-proof}.  In Section \ref{basic-sec} we prove the key proposition, Proposition \ref{basic}, that analyses the sets $S_r$ mentioned above, allowing us to quickly prove our main theorems in Section \ref{main-sec}.  In the final section, Section \ref{final-sec}, we specialise our theorems to specific rings such as division rings (or boundedly many products of division rings), cyclic groups, and algebras, to illustrate the results and also to recover some (but certainly not all) of the earlier sum-product results in the literature.

\subsection{Acknowledgments}

The author thanks Emmanuel Breuillard for posing this question, Jean Bourgain, Elon Lindenstrauss, Alex Gamburd, and Lior Silberman for useful discussions, Igor Shparlinski and Van Vu for some references, and an anonymous commenter and the anonymous referee for some corrections.  The author is supported by NSF grant DMS-0649473 and a grant from the Macarthur Foundation.

\section{Sum set estimates}\label{sumsetsec}

We recall some basic estimates from the Pl\"unnecke-Ruzsa theory of sum set estimates, as recorded for instance in \cite[Chapter 2]{tao-vu}.

\begin{lemma}[Ruzsa triangle inequality]\label{ruzsa-triangle}\cite{ruzsa}  If $A, B, C$ are finite non-empty subsets of an additive group $G$, then $|A-C| \leq |A-B| |B-C| / |B|$.
\end{lemma}

\begin{proof} From the identity $a-c = (a-b)+(b-c)$ we see that every element of $A-C$ has at least $|B|$ representations as the sum of an element of $A-B$ and an element of $B-C$.  The claim follows.
\end{proof}

\begin{lemma}[Ruzsa covering lemma]\label{ruzsa-cover}\cite{ruzsa-group}  If $A, B$ are finite non-empty subsets of an additive group $G$, then $A \subset B-B + \bigO(\frac{|A+B|}{|B|})$, and similarly $A \subset B-B + \bigO(\frac{|A-B|}{|B|})$.  (Recall that $\bigO(N)$ denotes an unspecified finite non-empty set of cardinality $O(N)$.)
\end{lemma}

\begin{proof}  Let $X$ be a maximal subset of $A$ with the property that the sets $x+B$ for $x \in X$ are disjoint.  One easily verifies that $|B| |X| \leq |A+B|$ and that $A \subset B-B+X$, and the first claim follows. The second claim follows by replacing $B$ by $-B$ (note that $B-B$ remains unchanged by this reflection).
\end{proof}

\begin{lemma}[Pl\"unnecke-Ruzsa sumset estimate]\label{plunruz}  Let $A, B$ be finite non-empty subsets of an additive group such that $|A+B| \leq K |A|$ and $B \geq |A|/K$. Then one has $|n_1 A - n_2 A + n_3 B - n_4 B| \leq K^{O_{n_1,n_2,n_3,n_4}(1)} |A|$ for all $n_1,n_2,n_3,n_4 \geq 0$.
\end{lemma}

\begin{proof} See e.g. \cite[Proposition 2.27]{tao-vu}.
\end{proof}

We will also routinely use elementary identities and inclusions in rings $R$ such as
\begin{align*}
A \cdot B \pm A \cdot C &\subset A \cdot (B \pm C)\\
a \cdot A \pm a \cdot B &= a \cdot (A\pm B)\\
(A \cdot B) \cdot C &= A \cdot (B \cdot C)\\
(A+B)+C &= A+(B+C)\\ 
A+B &= B+A \\
\bigO( N ) + \bigO( M ), \bigO(N) \cdot \bigO(M) &= \bigO( NM ) 
\end{align*}
for $A,B,C \subset R$, $a \in R$, and $N,M \geq 1$ without any further comment.

\begin{remark} There is also a non-commutative version of the above theory for use in multiplicative groups, see \cite{tao-noncomm}.  We will avoid using this theory here, though, since the multiplicative structure of a ring is not quite as strong as that of a group, even if we restrict to the cancellative semigroup $R^*$ of non-zero-divisors.  It seems of interest to develop such a theory for this semigroup, though.
\end{remark}

\subsection{Proof of Lemma \ref{ktl}}\label{ktl-proof}

We can now prove Lemma \ref{ktl}.  Let $A, K$ be as in that lemma.  From the identities
$$ |A|^2 = \sum_{x \in A \cdot A} |\{ (a,b) \in A \times A: ab = x \}|$$
and
$$ \sum_{x \in A \cdot A} |\{ (a,b) \in A \times A: ab = x \}|^2  = \sum_{a,b\in A} |a \cdot A \cap A \cdot b|$$ 
we conclude that
$$ \sum_{a,b \in A} |a \cdot A \cap A \cdot b| \geq \frac{|A|^4}{|A \cdot A|} \geq |A|^3/K.$$
Since $|a \cdot A \cap A \cdot b| \leq |A|$, we conclude that $|a \cdot A \cap A \cdot b| \geq |A|/2K$ for at least $|A|^2/2K$ pairs $(a,b)$.  By the pigeonhole principle again, we can find $b_0 \in A$ such that the set
$$ A' := \{ a \in A: |a \cdot A \cap A \cdot b| \geq |A|/2K \}$$
has cardinality at least $|A|/2K$.

Now let $a, a' \in A'$.  Then
$$ |(a \cdot A \cap A \cdot b) + A \cdot b| \leq |(A+A) \cdot b| \leq K |A|$$
and
$$ |(a \cdot A \cap A \cdot b) + a \cdot A| \leq |a \cdot (A+A)| \leq K |A|$$
and thus by the Ruzsa triangle inequality (Lemma \ref{ruzsa-triangle}) we have
$$ |a \cdot A - A \cdot b| \leq O( K^{O(1)} ) |A|$$
and thus by the covering lemma (Lemma \ref{ruzsa-cover})
\begin{equation}\label{aa}
a \cdot A \subset A \cdot b - A \cdot b + \bigO( K^{O(1)} ).
\end{equation}
Similarly
$$ a' \cdot A \subset A \cdot b - A \cdot b + \bigO( K^{O(1)} ).$$
Multiplying the latter by $a$ we conclude
$$ aa' \cdot A \subset a \cdot A \cdot b - a \cdot A \cdot b + \bigO( K^{O(1)} )$$
and hence by \eqref{aa}
we have
$$ aa' \cdot A \subset A \cdot b^2 - A \cdot b^2 + A \cdot b^2 - A \cdot b^2 + \bigO( K^{O(1)} )$$
and thus if $a'', a'''$ are also in $A'$ then
$$ (aa'-a''a''') \cdot A \subset (4A-4A) \cdot b^2 + \bigO( K^{O(1)} ).$$
In other words, for each $d \in A' \cdot A' - A' \cdot A'$ we have
$$ d \cdot A \subset (4A-4A) \cdot b^2 + \bigO( K^{O(1)} ).$$
By the pigeonhole principle, for each such $d$ we can find $x \in R$ and\footnote{Here and in the sequel we use $X \ll Y$ or $Y \gg X$ to denote the estimate $X \leq CY$ for some constant $C$.} $\gg K^{-O(1)} |A|$ elements $a$ in $A$ such that $da \in (4A-4A) \cdot b^2 + x$, and thus we can find $\gg K^{-O(1)} |A|^2$ pairs $a,a' \in A$ such that $d(a-a') \in (8A-8A) \cdot b^2$, and thus there are $\gg K^{-O(1)} |A|$ elements $f$ of $A-A$ such that $df \in (8A-8A) \cdot b^2$.  Since we may assume we are not in option (i) for some suitable choice of constants $c, C > 0$, we conclude that there are $\gg K^{-O(1)} |A|$ pairs $(f,g) \in (A-A)\cap R^* \times (8A-8A)$ such that $df = gb^2$.  But each pair $(f,g)$ can be associated to at most one $d$, thus
$$ |A' \cdot A' - A' \cdot A'| \ll K^{O(1)} |A-A| |8A-8A| / |A|$$
and hence by sumset estimates (Lemma \ref{plunruz})
$$ |A' \cdot A' - A' \cdot A'| \ll K^{O(1)} |A|$$
and the lemma follows.

\section{The key proposition}\label{basic-sec}

Our analysis of sets $A$ with good additive and multiplicative properties will hinge around the properties of certain very structured sets $S_r$ associated to each ring element $r \in R$, which roughly speaking corresponds to the ``dilate'' of $A$ (or of the ``completion'' of $A$) that contains $r$.  The precise structure theory we need is contained in the following proposition.

\begin{proposition}[Basic construction]\label{basic} If $C_0 > 0$ is a sufficiently large constant, then the following statements hold. Let $R$ be a ring, and let $A \subset R$ be a non-empty finite set.  Suppose that $|A \cdot A - A \cdot A| \leq K |A|$ and $|A \cdot A| \geq |A|/K$ for some $K \geq 2$, and suppose that $|(A-A) \backslash R^*| < K^{-C_0}|A|$.  For each $r \in R$, define the set
\begin{equation}\label{sr}
S_r := \{ x \in R: |x \cdot A + r \cdot A| \leq K^{C_0} |A| \}.
\end{equation}
\begin{itemize}
\item[(i)] (Self-improving property) If $r \in R^*$ and $x \in S_r$, then $|x \cdot A + r \cdot A| \leq K^{O(1)} |A|$.
\item[(ii)] (Size bound) If $r \in R^*$, then $|S_r| \leq K^{O(1)} |A|$; in particular, $S_r$ is finite.
\item[(iii)] (Additive structure) If $r \in R^*$, then $S_r$ is an additive group.
\item[(iv)] (Ring structure) If $R$ has an identity $1$, then $S_1$ is a ring.
\item[(v)] (Right-multiplicative structure) If $r \in R^*$ and $a \in (A-A) \cap R^*$, then $S_r \cdot S_a \subset S_{ra}$.
\item[(vi)] (Left-multiplicative structure) If $r,s \in R$ then $s \cdot S_r \subset S_{sr}$.
\item[(vii)] (Reflexivity) For every $r \in R$, we have $r \in S_r$.
\item[(viii)] (Symmetry) If $r,s \in R$, then $r \in S_s$ if and only if $s \in S_r$.
\item[(ix)] (Transitivity) If $r,s \in R^*$ and $S_r \cap S_s \cap R^* \neq \emptyset$, then $S_r = S_s$.
\end{itemize}
\end{proposition}

\begin{remark} The results here can be viewed as a refinement of the analysis of the ``good'' elements $x$ (for which $x \cdot A \subset A - A + \bigO(K^{O(1)})$) in \cite[Proposition 3.3]{bkt}.
\end{remark}

\begin{proof}  We first make the observation that if $|A \cdot A - A \cdot A| \leq K |A|$, then by sum set estimates (Lemma \ref{plunruz}) we have $|A \cdot A + A \cdot A - A \cdot A - A \cdot A| \leq K^{O(1)} |A|$ and thus $|(A-A) \cdot (A-A)| \leq K^{O(1)} |A|$.  Since $A-A$ contains at least one element in $R^*$, we conclude in particular that
\begin{equation}\label{Adiff}
|A-A| \leq K^{O(1)} |A|.
\end{equation}

We now show the (crucial\footnote{Indeed, the self-improving property allows us to absorb all losses of $K^{O(1)}$ which arise from our heavy use of the Pl\"unnecke-Ruzsa sumset theory.}) self-improving property (i).  Let $r \in R^*$, and let $x \in S_r$.  The map $(a_1,a_2) \to xa_1+ra_2$ maps $A \times A$ to a set of cardinality at most $K^{C_0} |A|$.  By Cauchy-Schwarz, we can thus find at least $K^{-C_0} |A|^3$ quadruplets $(a_1,a_2,a'_1,a'_2) \in A \times A \times A \times A$ such that $xa_1+ra_2 = xa'_1+ra'_2$, or equivalently $x(a_1-a'_1) = r(a'_2-a_2)$.  Since each difference $d \in A-A$ can be represented at most $|A|$ times in the form $a-a'$ for $a,a' \in A$, we conclude that there are at least $K^{-C_0} |A|$ elements $d, d' \in A-A$ such that $xd=rd'$; since $r$ is not a zero-divisor, we conclude that there are at least $K^{-C_0} |A|$ elements $d \in A-A$ such that $xd \in r \cdot (A-A)$.  Thus there exists at least one $d \in (A - A) \cap R^*$ such that $xd \in r \cdot (A-A)$.

Let $d$ be as above.  Then we have $xd\cdot A \in r\cdot (A \cdot A-A \cdot A)$.  On the other hand, as $d$ is not a zero-divisor, we have $|d \cdot A| = |A|$.  Since $|A \cdot A| \geq |A|/K$ and $|d \cdot A-A \cdot A| \subset |A \cdot A-A \cdot A - A \cdot A| \leq K^{O(1)} |A|$ (by Lemma \ref{plunruz}), we see from Ruzsa's covering lemma (Lemma \ref{ruzsa-cover}) that $A \cdot A \subset d \cdot A-d \cdot A + \bigO(K^{O(1)})$.  Thus
$$ x \cdot A \cdot A \subset r\cdot (A \cdot A-A \cdot A)-r\cdot (A \cdot A-A \cdot A) + \bigO(1) = r \cdot (2(A \cdot A) - 2(A \cdot A)) + \bigO(K^{O(1)})$$
and thus
$$ (x \cdot A + r \cdot A) \cdot (A-A) \subset r \cdot( 5(A \cdot A) - 5(A \cdot A)) + \bigO(K^{O(1)}).$$
By sum set estimates we conclude
$$|(x \cdot A + r \cdot A) \cdot (A-A)| \leq K^{O(1)} |A|.$$
Since $A-A$ contains at least one non-zero divisor, the claim (i) follows.

Now we prove the size bound (ii).  Let $r \in R^*$ and let $x \in S_r$.  By the self-improving property we have $|x \cdot A + r \cdot A| \leq K^{O(1)} |A|$.  Arguing as in the proof of (i), we have at least $K^{-O(1)}|A|$ elements $d \in A-A$ such that $xd \in r \cdot (A-A)$.  If $C_0$ is large enough, we can thus find at least $K^{-O(1)} |A|$ pairs $(d,d') \in (A-A) \times (A-A)$ such that $xd=rd'$ and $d$ is not a zero-divisor.  But each such pair $(d,d')$ corresponds to at most one $x$, while the total number of pairs $(d,d')$ is at most $|A-A|^2 \leq K^{O(1)} |A|^2$ by \eqref{Adiff}.  The claim (ii) follows.

Now we show (iii).  Let $r \in R^*$. If $x,y \in S_r$, then $|x \cdot A+r \cdot A| \leq K^{O(1)} |A|$ and $|y\cdot A+r \cdot A| \leq K^{O(1)} |A|$.  By the Ruzsa covering lemma (Lemma \ref{ruzsa-cover}) we conclude that 
\begin{equation}\label{xaa}
x \cdot A \subset r \cdot A- r \cdot A + \bigO(K^{O(1)})
\end{equation}
and similarly
\begin{equation}\label{yaa}
y \cdot A \subset r \cdot A- r \cdot A + \bigO(K^{O(1)}), 
\end{equation}
and thus $(x+y)\cdot A + r \cdot A \subset r\cdot(3A-2A) + \bigO(K^{O(1)})$.  Applying the sumset estimate (Lemma \ref{plunruz}), we obtain $|(x+y)\cdot A + r \cdot A| \leq K^{O(1)} |A|$, and so $x+y \in S$ if $C_0$ is large enough.  Thus $S$ is closed under addition; since it is also finite, it must be an additive group.  This establishes (iii).

Now we prove (iv).  In view of (iii), it suffices to show that $S_1$ is closed under multiplication.  Let $x, y \in S_1$, thus by (i) we have $|x \cdot A + A|, |y \cdot A + A| \leq K^{O(1)} |A|$.   Multiplying $y \cdot A + A$ by $x$ we have $|xy \cdot A + x \cdot A| \leq K^{O(1)} |A|$, and thus by Ruzsa's triangle inequality (Lemma \ref{ruzsa-triangle}) and sum set estimates (Lemma \ref{plunruz}) we have $|xy \cdot A + A| \leq K^{O(1)} |A|$, and thus (if $C_0$ is large enough) $xy \in S_1$, yielding the desired closure property.

Now we show (v).  Let $r \in R^*$ and $a \in (A-A) \cap R^*$, and let $x \in S_r$ and $y \in S_a$.
By (i), we have $|x \cdot A + r \cdot A| \leq K^{O(1)} |A|$ and $|y \cdot A + a \cdot A| \leq K^{O(1)} |A|$.  Also, we have
$$ a \cdot A + A \cdot a \subset A \cdot A - A \cdot A + A \cdot A - A \cdot A$$
and thus by sum set estimates
$$ |a \cdot A + A \cdot a| \leq K^{O(1)} |A|$$
and thus by the Ruzsa triangle inequality
$$ |y \cdot A + A \cdot a| \leq K^{O(1)} |A|$$
and thus by the Ruzsa covering lemma
$$ y \cdot A \subset A \cdot a - A \cdot a + \bigO( K^{O(1)} ).$$
Multiplying by $x$ we obtain
$$ xy \cdot A \subset x \cdot A \cdot a - x \cdot A \cdot a + \bigO( K^{O(1)} ).$$
Meanwhile, by the Ruzsa covering lemma again we have
$$ x \cdot A \subset r \cdot A - r \cdot A + \bigO( K^{O(1)} )$$
and thus
$$ xy \cdot A \subset r \cdot 2A \cdot a - r \cdot 2A \cdot a + \bigO( K^{O(1)} ).$$
But by the covering lemma again we have
$$ A \cdot a \subset a \cdot A - a \cdot A + \bigO( K^{O(1)} )$$
and thus
$$ xy \cdot A \subset ra \cdot 4A - ra \cdot 4A + \bigO( K^{O(1)} )$$
and thus
$$ |xy \cdot A + ra \cdot A| \leq K^{O(1)} |5A-4A|.$$
By sum set estimates we thus have $|xy \cdot A + ra \cdot A| \leq K^{O(1)} |A|$, and so $xy \in S_{ra}$.  The claim (v) follows. 

Clearly (vi) and (viii) are immediate from \eqref{sr}, while (vii) follows from \eqref{Adiff}, so we now turn to (ix).  Let $r,s,t \in R^*$ be such that $t \in S_r \cap S_s$.  Then by (i) we have $|t \cdot A + r \cdot A|, |t \cdot A + s \cdot A| \leq K^{O(1)} |A|$, and hence by the covering lemma (Lemma \ref{ruzsa-cover}) we have $r \cdot A \subset t \cdot A - t \cdot A + \bigO(K^{O(1)})$ and $t \cdot A \subset s \cdot A - s \cdot A + \bigO(K^{O(1)})$, and thus $r \cdot A \subset s \cdot (2A-2A) + \bigO(K^{O(1)})$.  If $x \in S_s$, we thus have
$$ |x \cdot A + r \cdot A| \leq K^{O(1)} |x \cdot A + s \cdot A + s \cdot A - s \cdot A - s \cdot A|$$
and hence by (i) and sum set estimates (Lemma \ref{plunruz})
$$ |x \cdot A + r \cdot A| \leq K^{O(1)} |A|$$
and thus $x \in S_r$.  This shows that $S_s \subset S_r$; a similar argument gives $S_r \subset S_s$, and the claim follows.
\end{proof}

\section{Proofs of theorems}\label{main-sec}

With the key proposition in hand, we can now quickly conclude the main theorems of the paper.

\subsection{Proof of Theorem \ref{sum1-thm}}

By increasing $K$ if necessary we may take $K \geq 2$.  Let $A$ be as in the theorem.  By sumset estimates (Lemma \ref{plunruz}) we have $|A \cdot A - A \cdot A| \leq K^{O(1)} |A|$; by increasing $K$ if necessary we can assume $|A \cdot A - A \cdot A| \leq K |A|$.

Let $C_0$ be a sufficiently large absolute constant.  We may assume that $|(A-A) \backslash R^*| < K^{-C_0} |A|$, since the claim is trivial otherwise.  We set $S := S_1$, where $S_r$ is defined in \eqref{sr}.  By Proposition \ref{basic}(ii) we have $|S| \leq K^{O(1)} |A|$; by Proposition \ref{basic}(iv) $S$ is a ring.  Finally, since $|A + A \cdot A| \leq K|A|$, we have $A \subset S$ by \eqref{sr}, if $C_0$ is large enough.  The claim follows.

\subsection{Proof of Theorem \ref{sum2-thm}}

Again we can take $K \geq 2$.  Let $A$ be as in the theorem with an invertible element $a$, and let $C_0$ be a sufficiently large absolute constant. Again we may assume $|(A-A) \backslash R^*| \leq K^{-C_0} |A|$.

We set $S := S_1$ again.  Using Proposition \ref{basic}(ii) as before, we have $|S| \leq K^{O(1)} |A|$, and from Proposition \ref{basic}(iv) $S$ is a ring.  

By sum set estimates (Lemma \ref{plunruz}), we have $|A \cdot A + A \cdot A| \leq K^{O(1)} |A|$.  Since $a \in A$, we conclude that $|A \cdot A + a \cdot A| \leq K^{O(1)} |A|$; since $a$ is invertible; we thus have $|a^{-1} \cdot A \cdot A + A| \leq K^{O(1)} |A|$, and thus (if $C_0$ is large enough) $a^{-1} \cdot A \subset S$, and thus $A \subset a \cdot S$.

The only remaining task is to show that $a\cdot S=S\cdot a$.  Since $S$ is finite and $a$ is invertible, it suffices to show that $a\cdot S\cdot a^{-1} \subset S$.  Now let $x \in S$.  Since $A \subset a\cdot S$, and $S$ is a ring, we have
$$ a x a^{-1} \cdot A + A \subset a \cdot S \cdot a^{-1} \cdot a \cdot S + a \cdot S = a S$$
and thus
$$ |a x a^{-1} \cdot A + A| \leq |S| \leq K^{O(1)} |A|$$
and thus $axa^{-1} \in S$ by \eqref{sr}.  This shows that $a \cdot S \cdot a^{-1} \subset S$ as claimed.

\begin{remark} It is also possible to deduce Theorem \ref{sum2-thm} from Theorem \ref{sum1-thm} by applying the latter theorem to $a^{-1} \cdot A$, though one does need to invoke the sum set estimates (in a manner very similar to that performed above) to verify that $a^{-1} \cdot A$ obeys the required hypotheses for a suitable choice of $K$.  We leave the details as an exercise to the reader.
\end{remark}

\subsection{Proof of Theorem \ref{sum3-thm}}

Again we can take $K \geq 2$.  Let $A$ be as in that theorem.
Let $C_0 > 0$ be a large absolute constant; as before we may assume that $|(A-A) \backslash R^*| < K^{-C_0} |A|$.  In particular we can find an element $a \in (A-A) \cap R^*$.  By hypothesis and \eqref{sr} we see that $A \subset S_a$.

For each $n \geq 1$, let $G_n := \langle A^n \rangle$ be the additive group generated by the $n$-fold product set $A^n := A \cdot \ldots \cdot A$.  Observe that $G_n \cdot G_m \subset G_{n+m}$ for all $n,m \geq 1$.  Also, by Proposition \ref{basic}(iii) we have $G_1 \subset S_a$, so by Proposition \ref{basic}(v) and induction we have $G_n \subset S_{a^n}$ for all $n \geq 1$.  In particular, from Proposition \ref{basic}(ii) we have
$$ |G_n| \leq K^{O(1)} |A|$$
for all $n$.  On the other hand, since $a \cdot G_n \subset G_{n+1}$, we know that $|G_n|$ is a non-decreasing function of $n$.  Thus the quantity $N := \lim_{n \to \infty} |G_n|$ is finite (and takes values between $|A|$ and $K^{O(1)}|A|$, and furthermore we have $|G_n|=N$ for all sufficiently large $n$.  In particular, we see that the map $x \mapsto ax$ is a bijection from $G_n$ to $G_{n+1}$ for all sufficiently large $n$.

For every integer $d \in \Z$, define a \emph{partial dilation} $T = (T_n)_{n=n_0}^\infty$ of degree $d$ to be a sequence of additive homomorphisms $T_n: G_n \to G_{n+d}$ defined for all $n \geq n_0$ and some $n_0 \geq \max(1,1-d)$ such that $T_{n+m}(g_n g_m) = T_n(g_n) g_m$ for all $n \geq n_0$, $m \geq 1$, $g_n \in G_n$, and $g_m \in G_m$.  We say that a partial dilation is \emph{maximal} if one cannot decrease $n_0$ (and add additional homomorphisms $T_n$) without destroying the partial dilation property.  We observe some key examples of partial and maximal dilations:

\begin{enumerate}
\item[(a)] The identity sequence $(g_n \mapsto g_n)_{n=1}^\infty$, where $g_n$ denotes a variable in $G_n$, is a maximal dilation of degree $0$.
\item[(b)] For any $m \geq 1$ and $h \in G_m$, the sequence $(g_n \mapsto h g_n)_{n=1}^\infty$ is a maximal dilation of degree $m$.  In particular every element of $A$ gives rise to a maximal dilation of degree $1$ in this manner.  Also this construction associates $0$ to a maximal dilation of degree $m$ for each $m$.
\item[(c)] For all $n$ larger than a sufficiently large constant $n_0$, we have seen that the map $g_{n-1} \mapsto ag_{n-1}$ from $G_{n-1}$ to $G_n$, and thus has an inverse $a g_{n-1} \mapsto g_{n-1}$.  The sequence $(ag_{n-1} \mapsto g_{n-1})_{n=n_0}^\infty$ is thus a partial dilation of degree $-1$.
\item[(d)] If $(T_n)_{n=n_0}^\infty$ and $(T'_n)_{n=n'_0}^\infty$ are partial dilations of the same degree $d$, then $(T_n + T'_n)_{n = \max(n_0,n'_0)}^\infty$ is also a partial dilation of degree $d$, where of course $T_n + T'_n: G_n \to G_{n+d}$ is the map $x \mapsto T_n x + T'_n x$.  Similarly, $(-T_n)_{n=0}^\infty$ is also a partial dilation of degree $d$.
\item[(e)] If $(T_n)_{n=n_0}^\infty$ and $(T'_n)_{n=n'_0}^\infty$ are partial dilations of degree $d$ and $d'$ respectively, then $(T_{n+d'} \circ T'_n)_{n=n_1}^\infty$ is a partial dilation of degree $d+d'$ for $n_1 := \max( n'_0, n_0 - d' )$.
\end{enumerate}

Suppose that we have two partial dilations $(T_n)_{n=n_0}^\infty$ and $(T'_n)_{n=n'_0}^\infty$ which collide in the sense that $T_{n_1} \equiv T'_{n_1}$ for some $n_1 \geq \max(n_0,n'_0)$.  We then claim that in fact $T_n \equiv T'_n$ for all $n \geq \max(n_0,n'_0)$.  To see this, we first observe from the identities $T_n(x) a^m = T_n(x a^m)$, $T'_n(x) a^m = T'_n(x a^m)$ and the fact that $a^m$ is not a zero divisor for any $m \geq 1$ that it suffices to establish this for sufficiently large $n$.  But when $n$ is large enough, we have $|G_n| = |G_{n-n_1}|=N$ and so the map $x \mapsto a^{n_1} x$ is a bijection from $G_{n-n_1}$ to $G_n$.  Since $T_{n_1}(a^{n_1}) = T'_{n_1}(a^{n_1})$, we conclude that $T_n \equiv T'_n$ as required.  We conclude in particular that every partial dilation has a unique maximal extension.

For each $d \in \Z$, let $R_d$ denote the collection of all maximal dilations of degree $d$.  From the previous paragraph and (d) we can give $R_d$ the structure of an additive group.  Also, from the previous paragraph and (e) we can define a product operation $\cdot: R_d \times R_{d'} \to R_{d+d'}$ which is associative and is distributive over the additive structure, thus the direct sum $\bigoplus_{d \in \Z} R_d$ has the structure of a (graded) ring; in particular $R_0$ is itself a ring.  From (a) we have a multiplicative identity $1 \in R_1$ in this ring, while from (b) we can embed $G_n$ into $R_n$ in a manner preserving the additive and multiplicative structure.  (The embedding is injective, as can be seen by testing the resulting dilations on a non-zero-divisor such as $a$.)  In particular, the element $a$ can be identified with an element $t$ of $R_1$.  Finally from the previous paragraph and (c) we see that we can construct an inverse $t^{-1} \in R_{-1}$ which is both a right and left inverse of $t$.  In particular, $R_d = R_0 t^d = t^d R_0$ for all $d \in \Z$.  We thus see that the ring $\bigoplus_{d \in \Z} R_d$ is isomorphic to the the ring $R_0(t)_\phi$, where $\phi: R_0 \to R_0$ is the outer automorphism $\phi: r_0 \mapsto t r_0 t^{-1}$, and we can embed $G^n$ with elements of $R_0 t^n$ in this ring.

Finally, we observe that if we have two elements $(T_n)_{n=n_0}^\infty$ and $(T'_n)_{n=n'_0}^\infty$ of $R_0$ such that $T_n(a^n) = T'_n(a^n)$ for some $n \geq \max(n_0,n'_0)$, then (since the map $x \mapsto a^n x$ is a bijection from $G_m$ to $G_{m+n}$ for sufficiently large $m$) we conclude that $T_{m+n} \equiv T'_{m+n}$ for sufficiently large $m$, and thus the maximal extensions $(T_n)_{n=n_0}^\infty$ and $(T'_n)_{n=n'_0}^\infty$ must be identical.  On the other hand, $T_n(a^n)$ takes values in a set of size at most $N$.  We conclude that $R_0$ is finite with cardinality at most $N \leq K^{O(1)} |A|$.  (In fact, since $R_0$ contains $t^{-n} G_n$ for every $n \geq 1$, we see that $R_0$ has cardinality exactly $N$.)  The claim follows.

\begin{remark} In the case when the original ring $R$ is commutative, one can show that all maximal dilations commute with each other, so that $R_0$ is commutative and the twist map $\phi$ used to define $R_0(t)_\phi$ is in fact trivial; we leave the verification of this as an exercise to the reader.  Thus in this case one can embed $A$ into the commutative polynomial ring $R_0[t]$.
\end{remark}

\section{Special cases}\label{final-sec}

We now specialise the above theory to various special cases of interest.  Broadly speaking, our results are useful in any context in which the set of zero divisors is sparse and has an easily understood structure; this covers many (but definitely not all) cases of interest.

\subsection{Division rings}

The simplest application is to division rings, since in this case every non-zero element is invertible (and thus not a zero-divisor):

\begin{theorem}[Sum-product phenomenon in division rings]\label{division-thm}  Let $D$ be a division ring, and let $A$ be a finite non-empty subset of $D$ such that $|A+A|, |A \cdot A| \leq K|A|$ for some $K \geq 1$.  Then there exists a finite subring $S$ of $D$ of cardinality $|S| \ll K^{O(1)} |A|$ and an invertible element $a \in A$ such that $a \cdot S = S \cdot a$ and $A \subset a \cdot S + \bigO( K^{O(1)} )$.
\end{theorem}

For the case when $D$ is commutative (i.e. $D$ is a field, this is \cite[Theorem 2.55]{tao-vu}).  For finite-dimensional division rings over $\R$, such as the quaternions, this result is in \cite{chang}.  Note that this result also implies the original sum-product result of Erd\H{o}s and Szemer\'edi \cite{erdos-sumproduct} in $\Z$, as well as the sum-product result over finite fields in \cite{bkt}, \cite{konyagin}, \cite{konyagin2}.  One also recovers the sum-product estimates from \cite{wood} for integral domains of characteristic zero, since these domains can be embedded inside their field of fractions and contain no finite subrings.

\begin{proof}  We can assume that $|A| \geq C K^C$ for some large $C$, since the claim is trivial otherwise.  Applying Lemma \ref{ktl} (removing $0$ from $A$ if necessary) we can find a subset $A'$ of $A$ with $|A'| \gg K^{-O(1)} |A|$ such that $|A' \cdot A' - A' \cdot A'| \ll K^{O(1)} |A'|$ (since there are not enough zero divisors for option (i) to hold).  Applying Theorem \ref{sum2-thm} with $a$ equal to an arbitrary non-zero (and hence invertible) element of $A'$, we conclude that there exists a finite subring $S$ of $D$ of cardinality $|S| \ll K^{O(1)} |A'| \leq K^{O(1)} |A|$ with $a \cdot S = S \cdot a$ such that $A' \subset a \cdot S$.  Since $|A + A'| \leq |A + A| \leq K |A|$, the claim then follows from the covering lemma (Lemma \ref{ruzsa-cover}).
\end{proof}

\subsection{Products of division rings}

After division rings, the next easiest case to study is the product $R = D_1 \times \ldots \times D_k$ of a bounded number $k=O(1)$ of division rings (with the obvious pointwise ring operations), since in this case the zero divisors of $R$ are easily identified and have a clean and sparse structure, indeed we have $R^* = (D_1 \backslash \{0\}) \times \ldots \times (D_k \backslash \{0\})$.  A model example here is $\F_p \times \F_p$; the sum-product phenomenon in this ring was first studied by Bourgain \cite{borg-diffie-comp}, \cite{borg-diffie}, \cite{borg-mordell} in connection with exponential sums connected to the Diffie-Hellman cryptosystem, and also to certain exponential sums of Mordell type.

Our main result here is as follows.

\begin{theorem}[Sum-product phenomenon in products of division rings]\label{prod-thm} Let $D_1,\ldots,D_k$ be division rings, let $R := D_1 \times \ldots \times D_k$, and let $A \subset R$ be a non-empty finite set such that $|A+A|, |A \cdot A| \leq K|A|$ for some $K \geq 1$.  Then one of the following holds:
\begin{itemize}
\item[(i)] There exists $1 \leq j \leq k$ such that $|\pi_j(A)| \ll k K^{O(1)}$, where $\pi_j: R \to D_j$ is the projection homomorphism.
\item[(ii)] There exists a finite subring $S$ of $R$ of cardinality $|S| \ll K^{O(1)} |A|$ and an invertible element $a \in A$ such that $a \cdot S = S \cdot a$ and $A \subset a \cdot S + \bigO( K^{O(1)} )$.
\end{itemize}
\end{theorem}

This result implies the sum-product theorems in \cite{borg-diffie-comp}, \cite{borg-diffie}, \cite{borg-mordell} for the ring $\F_p \times \F_p$ as a special case, after noting that the only non-trivial proper subrings of $\F_p \times \F_p$ are of the form $\F_p \times \{0\}$, $\{0\} \times \F_p$, or $\{ (x,ax): x \in \F_p \}$ for some $a \in \F_p^*$.  Note that the results cited treat the case when $K \geq p^\eps$ for $\eps > 0$, but Theorem \ref{prod-thm} can also be applied for much smaller values of $K$ (though this is not the case of interest for exponential sum applications).  It also largely recovers the sum-product estimates in \cite{bourgain-chang-zq}, \cite{bourgain-survey}.  

For applications to exponential sums, the above theorem is only useful when $k$ is small; when $k$ is large, the conclusion (i) becomes weak unless $K$ is small compared to $\min_{1 \leq j \leq k} |D_j|$, which is not the case of interest in these applications.  In particular, we do \emph{not} recover the sum-product theorem for $\Z_p^k$ for fixed $p$ and large $k$ that appears in \cite{borg-finite-exp}, \cite{borg-sump-exp}, \cite{borg-exponential}. 

\begin{proof}  Suppose first that $A-A$ contains $\gg K^{-O(1)} |A|$ zero-divisors, then by the pigeonhole principle there exists $1 \leq j \leq k$ such that the set $A_j := (A-A) \cap \pi_j^{-1}(\{0\})$ has cardinality $\gg K^{-O(1)} |A| / k$.  Observe that $|A + A_j| \geq |\pi_j(A)| |A_j|$ due to the disjointness of the fibres $\pi_j^{-1}(\{a_j\})$ for $a_j \in D_j$.  On the other hand, by sum set estimates (Lemma \ref{plunruz}) we have
$$ |A_j + A| \leq |A-A+A| \ll K^{O(1)} |A|$$
and thus
$$ |\pi_j(A)| \ll k K^{O(1)}$$
and then we are in case (i).

We may thus assume instead that $A-A$ does not contain so many zero-divisors.  The above argument also shows that we may assume that $|A_j| \leq c K^{-C} |A|/k$ for any fixed $c,C > 0$ and all $j$, which implies that $|A \cap \pi_j^{-1}(\{0\})| \leq c K^{-C} |A|/k$  In particular, at most half of the elements of $A$ are zero-divisors.  By removing such elements (and increasing $K$ slightly) we may thus assume that $A \subset R^*$.  We can now apply Lemma \ref{ktl} to find a subset $A' \subset A$ with $|A'| \gg K^{-O(1)} |A|$ such that $|A' \cdot A' - A' \cdot A'| \ll K^{O(1)} |A'|$.  Applying Theorem \ref{sum2-thm} (and noting that $A'-A'$ is a subset of $A-A$ and thus cannot have $\gg K^{-O(1)} |A'|$ zero-divisors) we conclude that there exists a finite subring $S$ of $R$ of cardinality $|S| \ll K^{O(1)} |A|$ and $a \in A'$ such that $a \cdot S = S \cdot a$ and $A' \subset a \cdot S$.  The claim then follows from the covering lemma as in the proof of Theorem \ref{division-thm}.
\end{proof}

\subsection{Cyclic rings of low prime power order}

Another interesting case are the cyclic $\Z/p^k\Z$ of prime power order.  We are not able to obtain satisfactory results in the important case when $k$ is large (in particular, we do not recover the results in \cite{borg-sump-exp}, \cite{borg-exponential}), but we can obtain the following result which is efficient in the regime $k=O(1)$.

\begin{theorem}[Sum-product phenomenon in cyclic rings]\label{cyclic-thm} Let $p$ be a prime, let $R := \Z/p^{k} \Z$ for some $k \geq 1$, and let $A \subset R$ be a non-empty finite set such that $|A+A|, |A \cdot A| \leq K|A|$ for some $K \geq 1$.  Then one of the following holds:
\begin{itemize}
\item[(i)] We have $A \subset p \cdot R + \bigO( K^{O(1)} )$;
\item[(ii)] We have $|A| \gg K^{-O(1)} |R|$.
\end{itemize}
\end{theorem}

\begin{proof}  By using the covering lemma (Lemma \ref{ruzsa-cover}), we see that if the set $A_1 := (A-A) \cap p \cdot R$ has cardinality $\gg K^{-O(1)} |A|$, then $A \subset p \cdot R + \bigO( K^{O(1)} )$. Thus we may assume that $|A_1| \leq c K^{-C} |A|$ for some suitable $c, C > 0$.  In particular this implies that at most half the elements of $A$ lie in $p \cdot R$, so by removing those elements and increasing $K$ slightly we may assume that all elements of $A$ are invertible.  We then apply Lemma \ref{ktl} to find a subset $A'$ of $A$ of cardinality $\gg K^{-O(1)} |A|$ such that $|A' \cdot A' - A' \cdot A'| \ll K^{O(1)} |A'|$; by theorem \ref{sum2-thm}, we can thus find $a \in A'$ and a subring $S$ of $R$ such that $|S| \ll K^{O(1)}|A|$ and $A' \subset a \cdot S$.  But the only subring of the cyclic ring $R$ that contains an invertible element is the full ring $R$, and so we are in case (ii) as desired.
\end{proof}

\subsection{Algebras}

Now we consider the case of a finite-dimensional algebra $R$ over a field $F$ (which may be finite or infinite), such as a matrix algebra $M_d(F)$.  As before, our results are only useful in the regime when the dimension $d$ is relatively low and $F$ is large.  More precisely, our result is as follows.

\begin{theorem}[Sum-product phenomenon in algebras]\label{algebra}  Let $R$ be a $d$-dimensional algebra over some field $F$, and let $A \subset R$ be a non-empty set with $|A+A|, |A \cdot A| \leq K|A|$ for some $K \geq 1$.  Suppose that the characteristic of $F$ is either zero, or is sufficiently large depending on $d$.  Then one of the following holds:
\begin{itemize}
\item[(i)] There exists a linear subspace $V \subset R$ such that every element of $V$ is a zero divisor, and such that $|A \cap (x+V)| \gg_d K^{-O_d(1)} |A|$ for some $x \in R$.  (The subscripting by $d$ means that the implied constants are allowed to depend on $d$.)
\item[(ii)] There exists subset $A'$ of $A$ of cardinality $\gg K^{-O(1)} |A|$, a finite ring $R_0$, an automorphism $\phi: R_0 \to R_0$, and embeddings $\iota_n: \langle (A')^n \rangle \to R_0 \cdot t^n \subset R_0[t]_\phi$ for $n \geq 0$ which are additive homomorphisms, and such that $\iota_{n+m}(a_n a_m) = \iota_n(a_n) \iota_m(a_m)$ for all $n,m \geq 0$ and $a_n \in \langle (A')^n \rangle$, $a_m \in \langle (A')^m \rangle$.
\end{itemize}
\end{theorem}

The author conjectures that the hypothesis on the characteristic can be omitted, but does not know how to prove this.

Before we show Theorem \ref{algebra}, we first need a result combining algebraic geometry with additive combinatorics, which may be of some independent interest.

\begin{lemma}[Linearisation]\label{linear}  Let $V$ be a $d$-dimensional vector space over a field $F$, let $W$ be an algebraic set $\{ x \in V: P_1(x) = \ldots = P_k(x)=0 \}$ cut out by $k$ polynomials $P_1,\ldots,P_k: V \to F$ of degree at most $D$, and let $A \subset W$ be a non-empty set such that $|A + A| \leq K|A|$.   Suppose that the characteristic of $F$ is either zero, or is sufficiently large depending on $d,k,D$. Then there exists an affine space $U$ contained in $W$ such that $|A \cap U| \gg_{d,k,D} K^{-O_{d,k,D}(1)} |A|$.
\end{lemma}

\begin{proof}  Without loss of generality we may take $F$ to be algebraically closed, since the general case follows by replacing $F$ with its algebraic closure $\overline{F}$, and $V$ with the tensor product $V \otimes_F \overline{F}$ (and noting that the restriction of an affine space over $\overline{F}$ in $V \otimes_F \overline{F}$ to $V$ is an affine space over $F$).  

Since any algebraic set cut out by $k$ polynomials of degree at most $D$ can be expressed\footnote{This is Corollary 6.11 of Kleiman's article~\cite{klei:sga6} in SGA6} as the union of $O_{d,k,D}(1)$ algebraic varieties (i.e. irreducible algebraic sets) cut out by $O_{d,k,D}(1)$ polynomials of degree at most $O_{d,k,D}(1)$, it suffices (by the pigeonhole principle) to verify the claim for algebraic varieties.

We induct on the dimension $d$; since the claim is trivial for $d=0$, we assume $d \geq 1$ and that the claim has been proven for smaller $d$.  Once $d$ is fixed, we perform a secondary induction on the dimension $\dim(W)$ of the algebraic variety.  The claim is trivial if $\dim(W)=0$ (since $|W| = O_{d,k,D}(1)$ in this case), so assume $\dim(W) \geq 1$ and that the claim has already been proven for smaller dimensional varieties in the same vector space $V$.

Suppose that there is a non-zero element $v$ of $V$ such that $W+v = W$.  Let $I(W)$ be the ideal of polynomials that vanish on $W$, then this ideal is invariant under translation by $v$.  In particular, if $P$ is a polynomial in $I(W)$, then the derivative $P(\cdot+v) - P(\cdot)$ also lies in $I(W)$.  Iterating this (starting with one of the generators of $W$) we conclude that $I(W)$ contains at least one non-trivial polynomial $P$ of degree $O_{d,k,D}(1)$ that is invariant under $v$; from the assumption on the characteristic of $F$, we conclude that $P$ factors through the projection map $\pi: V \to V/F \cdot v$.  As a consequence, we see that the projection $\pi(W)$ is a non-trivial algebraic set in the vector space $V/F \cdot v$, which is cut out by $O_{d,k,D}(1)$ polynomials of degree at most $O_{d,k,D}(1)$.  

Let $M$ be the maximum value of $|A \cap \pi^{-1}(\{x\})|$ as $x$ ranges over $V/F \cdot v$.  Observe that $|A+A| \geq |\pi(A)| M$; since $|A+A| \leq K|A|$, we conclude that $\sum_{x \in \pi(A)} |A \cap \pi^{-1}(\{x\})| \leq M |\pi(A)|/K$.  We conclude that there exists a subset $\tilde A \subset \pi(A)$ of cardinality $|\tilde A| \geq |\pi(A)|/2K$ such that $|A \cap \pi^{-1}(\{x\})| \geq M/2K$ for all $x \in \tilde A$.  If we now decompose $\pi(W)$ into algebraic varieties and use the pigeonhole principle as before, and apply the outer induction hypothesis, we conclude that there exists an affine space $\tilde U$ contained in $\pi(W)$ such that $|\tilde A \cap \tilde U| \gg_{d,k,D} K^{-O_{d,k,D}(1)} |\tilde A|$.  If we let $U := \pi^{-1}(\tilde U)$ we see from construction that $U$ is an affine space contained in $W$ and
\begin{align*} |A \cap U| &\geq |\tilde A \cap \tilde U| M/2K \\
&\gg_{d,k,D} K^{-O_{d,k,D}(1)} |\tilde A| M \\
&\gg_{d,k,D} K^{-O_{d,k,D}(1)} |\pi(A)| M \\
&\gg_{d,k,D} K^{-O_{d,k,D}(1)} |A|
\end{align*}
thus closing the induction in this case.  Thus we may assume without loss of generality that $W+v \neq W$ for all non-zero $v$; since $W$ is irreducible, this implies that $W+v \cap W$ is an algebraic set consisting of varieties of strictly smaller dimension than $W$.

Using the identities
$$ \sum_{a,a' \in A} |(a + A) \cap (a' + A)| = \sum_{x \in A+A} |\{ (a,b) \in A: a+b = x \}|^2$$
and
$$ \sum_{x \in A+A} |\{ (a,b) \in A: a+b = x \}| = |A|^2$$
and Cauchy-Schwarz, we have
$$ \sum_{a,a' \in A} |(a + A) \cap (a' + A)| \geq |A|^4/|A+A| \geq |A|^3/K.$$
The contribution of the diagonal term $a=a'$ is negligible unless $|A| = O(K)$, in which case the claim is trivial.  In all other cases, we can apply the pigeonhole principle and conclude that there exist distinct $a,a' \in A$ such that
$$ |(a + A) \cap (a' + A)| \gg |A|/K.$$
If we set $v := a'-a$ and $A' := A \cap (A+v)$, we thus conclude that $|A'| \gg |A|/K$ and 
$$|A'+A'| \leq|A+A| \leq K|A| \ll K^2 |A'|.$$
Also, $A'$ is contained in $W \cap (W+v)$ and is thus contained in $O_{d,k,D}(1)$ varieties of dimension strictly less than $W$, and cut out by $O_{d,k,D}(1)$ polynomials of degree $O_{d,k,D}(1)$.  Applying the inner induction hypothesis we may thus find an affine subspace $U$ in $W \cap (W+v)$ (and hence in $W$) such that $|A' \cap U| \gg_{d,k,D} K^{-O_{d,k,D}(1)} |A'| \gg_{d,k,D} K^{-O_{d,k,D}(1)} |A|$, thus closing the induction in this case also.
\end{proof}

Now we can prove Theorem \ref{algebra}.

\begin{proof}[Proof of Theorem \ref{algebra}]  
We can assume that $|A| \geq C K^C$ for some large $C$, since the claim is trivial otherwise. By writing the product operation on $R$ in coordinates over $F$, we observe that the set of non-zero-divisors in $R$ is an algebraic set cut out by $O_d(1)$ polynomials of degree at most $O_d(1)$.  If more than half the elements of $A$ lie in this set, then we can apply Lemma \ref{linear} to those elements and establish conclusion (i) of the theorem; thus we may assume that fewer than half of the elements of $A$ are zero-divisors.  By throwing away all the zero-divisors we may thus assume that $A$ lies entirely in $R^*$. Applying Lemma \ref{ktl} and we may now assume that either $A-A$ contains $\gg K^{-O(1)} |A|$ zero-divisors, or that $|A' \cdot A' - A' \cdot A'| \ll K^{O(1)} |A|$ for some subset $A' \subset A$ of cardinality $\gg K^{-O(1)} |A|$.  In the former case, if we apply Lemma \ref{linear} to $(A-A) \backslash R^*$ we can locate a vector space $V$ of zero divisors such that the set $A' := (A-A) \cap V$ has cardinality $\gg_d K^{-O_d(1)} |A|$ elements of $A-A$.  By sum set estimates (Lemma \ref{plunruz}) we have
$$ |A' + A| \leq |A-A+A| \ll K^{O(1)} |A|$$
and so by the covering lemma (Lemma \ref{ruzsa-cover}) we conclude that
$$ A \subset A'-A' + \bigO_d(K^{O_d(1)}) \subset V + \bigO_d(K^{O_d(1)})$$
and so we are again in conclusion (i) of the theorem.  Thus we may assume that $|A' \cdot A' - A' \cdot A'| \ll K^{O(1)} |A'|$ for some $A' \subset A$ of cardinality $\gg K^{-O(1)} |A|$.  We then apply Theorem \ref{sum3-thm} and conclude that either $A'-A'$ contains $\gg K^{-O(1)} |A|$ zero divisors, or else we are in conclusion (ii) of the theorem.  In the former case we can argue as before to end up in conclusion (i) of the theorem, and the claim follows.
\end{proof}

It should be possible to analyse the conclusions (i) and (ii) of Theorem \ref{algebra} further for specific algebras.  For instance, when $F$ is finite and $R$ contains an identity, then every non-zero-divisor is invertible, and one can use Theorem \ref{sum2-thm} instead of Theorem \ref{sum3-thm} to simplify conclusion (ii).  If instead we have $R = M_2(F)$, then the only affine spaces of zero divisors are either one or zero-dimensional, or are two dimensional and consist of the matrices which either left-annihilate or right-annihilate a nonzero vector $v \in F^2$; this leads to a more explicit description of conclusion (i), although the final form is somewhat complicated to express and will not be done here.

\end{document}